\documentclass[a4paper,reqno,10pt]{amsart}

\usepackage{amsmath}
\usepackage{amsthm}
\usepackage{amssymb}
\usepackage{amscd} 
\usepackage{bbm} 
\usepackage{mathrsfs} 
\usepackage{lmodern} 
\usepackage[T1]{fontenc} 
\usepackage{newtxtext,newtxmath}

\usepackage{manfnt} 

\usepackage{graphicx} 

\usepackage{verbatim} 

\usepackage[Lenny]{fncychap} 

\usepackage{sqrcaps} 

\usepackage{comment}

\usepackage[svgnames]{xcolor}
\usepackage{pdfcolmk}

\usepackage[hidelinks]{hyperref}

\usepackage{tikz-cd} \usepackage{pifont}

\swapnumbers 

\newtheoremstyle{exercise} 
  {3pt} 
  {3pt} 
  {\small\rmfamily} 
  {
} 
  {\rmfamily\scshape} 
  {.} 
  {.5em} 
  {} 

\newtheoremstyle{newplain}
  {5pt}
  {5pt}
  {\itshape}
  {}
  {\rmfamily\scshape}
  {. ---}
  {.5em}
  {}

\newtheoremstyle{newremark}
  {5pt}
  {5pt}
  {\rmfamily}
  {}
  {\rmfamily\scshape}
  {. ---}
  {.5em}
  {}

\theoremstyle{newplain}
\newtheorem*{Theorem*}{Theorem} 

\theoremstyle{newplain}
\newtheorem{Theorem}{Theorem}
\newtheorem{Lemma}[Theorem]{Lemma}
\newtheorem{Corollary}[Theorem]{Corollary}
\newtheorem{Proposition}[Theorem]{Proposition}
\newtheorem{Conjecture}[Theorem]{Conjecture}
\newtheorem{Definition}[Theorem]{Definition}

\theoremstyle{newremark}
\newtheorem{Empty}[Theorem]{}

\newtheorem{Claim}[Theorem]{Claim}
\newtheorem*{Claim*}{Claim}
\theoremstyle{exercise}

\numberwithin{Theorem}{section}
\numberwithin{Exercise}{section}

\newcommand{\R}{\mathbb{R}} 

\newcommand{\diff}{\mathop{}\mathopen{}\mathrm{d}} 

\newcommand{\calD}{\mathscr{D}}

\newcommand{\calL}{\mathscr{L}}

\newcommand{\bCH}{\mathbf{CH}} 

\newcommand{\bF}{\mathbf{F}}

\newcommand{\bM}{\mathbf{M}}
\newcommand{\bN}{\mathbf{N}}

\newcommand{\bP}{\mathbf{P}}

\newcommand{\rmsch}{\operatorname{sch}}

\DeclareMathOperator{\rmdist}{\mathrm{dist}} 

\DeclareMathOperator{\rmLip}{\operatorname{Lip}} 

\DeclareMathOperator{\rmloc}{\mathrm{loc}}

\DeclareMathOperator{\rmspt}{\operatorname{spt}} 

\newcommand{\hel} {
\hskip2.5pt{\vrule height7pt width.5pt depth0pt}
\hskip-.2pt\vbox{\hrule height.5pt width7pt depth0pt}
\, }

\def\XXint#1#2#3{{%
\setbox0=\hbox{$#1{#2#3}{\int}$}
\vcenter{\hbox{$#2#3$}}\kern-.5\wd0}}

\renewcommand{\leq}{\leqslant}
\renewcommand{\geq}{\geqslant}
\renewcommand{\subset}{\subseteq}

\usepackage{newtxtext}
\usepackage{newtxmath}
\DeclareMathAlphabet\euscr{U}{eus}{m}{n}
\newcommand{\niceBV}{\euscr{BV}}

\newcommand{\defeq}{\mathrel{\mathop:}=}

\begin{document}



\title{On the exterior product of Hölder differential forms}

\author[Ph. Bouafia]{Philippe Bouafia}

\address{F\'ed\'eration de Math\'ematiques FR3487 \\
  CentraleSup\'elec \\
  3 rue Joliot Curie \\
  91190 Gif-sur-Yvette
}

\email{philippe.bouafia@centralesupelec.fr}


\begin{abstract}
  We introduce a complex of cochains, $\alpha$-fractional charges ($0
  < \alpha \leq 1$), whose regularity is between that of De
  Pauw-Moonens-Pfeffer's charges and that of Whitney's flat
  cochains. We show that $\alpha$-Hölder differential forms and their
  exterior derivative can be realized as $\alpha$-fractional charges,
  and that it is possible to define the exterior product between an
  $\alpha$-fractional and a $\beta$-fractional charge, under the
  condition that $\alpha + \beta > 1$. This construction extends the
  Young integral in arbitrary dimension and codimension.
\end{abstract}

\maketitle
\tableofcontents

\section{Introduction}

The Young integral has become a classical object: given two $\alpha$-
and $\beta$-Hölder continuous functions $f$ and $g$ on the interval
$[0, 1]$, the convergence of the Riemann-Stieltjes integral $\int_0^1
f \diff g$ is assured when $\alpha + \beta > 1$. R. Züst proposed a
higher-dimensional extension in \cite[Section~3]{Zust}, introducing a
Riemann-Stieltjes type integral to handle expressions like
\begin{equation}
\label{eq:zust}
\int_{[0, 1]^d} f \diff g_1 \wedge \cdots \wedge \diff g_d
\end{equation}
where $f, g_1, \dots, g_d$ are Hölder functions of exponents
$\alpha_0, \alpha_1, \dots, \alpha_d$ satisfying $\alpha_0 + \cdots +
\alpha_d > d$. This condition is proven to be sharp, as
counterexamples are provided in the critical case
\cite[3.2]{Zust}. Furthermore, in \cite{Boua}, an extension of the
Züst integral is proposed, allowing the integrator to be a Hölder
charge. This generalization also enabled to define pathwise integrals
with respect to stochastic processes, such as the fractional Brownian
sheet with sufficient regularity \cite{BouaDePa}.

This goal of this article is to define integrals of Hölder
differential forms, such as those depicted in~\eqref{eq:zust}, in any
codimension, allowing the domain of integration to be a normal
current. In continuation of~\cite{Boua}, we are led to consider the
more general formalism of charges in middle dimension, as introduced
by Th. De Pauw, L. Moonens and W. Pfeffer in~\cite{DePaMoonPfef}, to
represent the generalized differential forms Hölder forms must be.  An
$m$-charge is a linear functional on the space of normal currents that
satisfies a certain continuity condition. Equivalently, the
representation theorem of charges \cite[Theorem~6.1]{DePaMoonPfef}
states that $m$-charges are the sums $\omega + \diff \eta$ of a
continuous $m$-form and the weak exterior derivative of an
$(m-1)$-form. The action on normal currents is given by the formula
\[
(\omega + \diff \eta)(T) = \int_{\R^d} \langle \omega(x), \vec{T}(x)
\rangle \diff \|T\|(x) + \int_{\R^d} \langle \eta(x),
\overrightarrow{\partial T}(x) \rangle \diff \| \partial T \|(x).
\]
Of course, there is no meaningful definition of the exterior product
of two general charges. This obstacle, which is essentially of
distributional nature, prevents us from defining integrals as
in~\eqref{eq:zust} in general.

The situation stands in stark contrast to the complex of locally flat
cochains introduced by H. Whitney, see \cite{Whit, FedeReal}. These
cochains are the continuous linear functionals on the space of flat
$m$-chains. The main result of the theory is Wolfe's representation
theorem, which asserts that locally flat $m$-cochains are in
one-to-one correspondence with locally flat $m$-forms:
$L^\infty_{\rmloc}$ differential $m$-forms with $L^\infty_{\rmloc}$
weak exterior derivative. Consequently, this allows for a pointwise
definition of the exterior product of two locally flat cochains.

However, charges do not possess the sufficient regularity, and it is
evident that the Hölder forms appearing in~\eqref{eq:zust} cannot be
assigned a pointwise meaning. Besides, the De~Pauw-Moonens-Pfeffer
representation theorem only allows to define a cup product at the
cohomology level, see \cite[Lemma~8.2]{DePaMoonPfef}. Thus, to develop
a full exterior calculus apparatus for charges, it is necessary to add
some regularity assumptions. We propose the following definition: an
$m$-charge $\omega$ over $\R^d$ is said to be $\alpha$-fractional
whenever, for every compact $K \subset \R^d$, one can find $C_K \geq
0$ such that
\[
|\omega(T)| \leq C_K \bN(T)^{1 - \alpha} \bF(T)^{\alpha}
\]
for every normal current $T$ supported in $K$. This definition is
designed so that $\alpha$-Hölder continuous differential forms and
their exterior derivatives can be represented as fractional
charges. In the zero-codimensional case, it coincides with the Hölder
charges described in~\cite{BouaDePa}. When $\alpha = 1$ we recover the
definition of Whitney's locally flat cochains.

The main accomplishment of this paper is the definition of the
exterior product between an $\alpha$-fractional charge $\omega$ and a
$\beta$-fractional charge $\eta$, under the Young type condition that
$\alpha + \beta > 1$. The resulting charge is also fractional, with
the fractional exponent being $\alpha + \beta - 1$. Our methods are
inspired by tools from harmonic analysis and bear resemblance to the
Fourier approach to integration developed by M. Gubinelli,
P. Imkeller, and N. Perkowski in \cite{GubiImkePerk}. Our principal
tool is a Littlewood-Paley type decomposition result for fractional
charges, that constitutes a generalization (over $\R^d$ only) of the
decomposition of Hölder functions described in
\cite[Appendix~B,~2.6]{Grom}. By introducing the Littlewood-Paley
decompositions of $\omega$ and $\eta$, it is then possible to split
$\omega \wedge \eta$ formally into two paraproducts, the existence of
which is easier to establish.

The paper is structured into six sections, each with self-explanatory
titles. Section~\ref{sec:charges} provides a self-contained
introduction to charges in middle dimension, focusing solely on the
results pertinent to this paper.

\section{Preliminaries}

\begin{Empty}[Notations]
  Throughout the article, the ambient space will be $\R^d$, with $d
  \geq 1$. It is equipped with the Lebesgue outer measure, denoted
  $\calL^d$. A measurable subset of $\R^d$ always refers to a set that
  is Lebesgue-measurable.

  For a function $f \colon X \to E$ defined on a locally closed subset
  $X \subset \R^d$ with values in a normed space $(E, \| \cdot \|)$, a
  subset $Y \subset X$ and $0 < \alpha \leq 1$, we define the extended
  real numbers
  \begin{align*}
    \| f \|_{\infty, Y} & \defeq \sup \left\{ \|f(y)\| : y \in Y
    \right\} \\ \rmLip^\alpha(f; Y) & \defeq \sup \left\{
    \frac{\|f(y_1) - f(y_2)\|}{|y_1 - y_2|^\alpha} : y_1, y_2 \in Y
    \text{ and } y_1 \neq y_2 \right\}
  \end{align*}
  In addition, we write $\| f \|_\infty = \|f\|_{\infty, X}$ and
  $\rmLip^\alpha(f) = \rmLip^\alpha(f; X)$. When $\alpha = 1$, we may
  of course write $\rmLip$ instead of $\rmLip^\alpha$.
  
  The space of $\alpha$-Hölder continuous maps is written
  $\rmLip^\alpha(X; E)$. A function $f \colon X \to E$ is locally
  $\alpha$-Hölder continuous whenever $\rmLip^\alpha(f; K) < \infty$
  for all compact $K \subset X$, and the space of such functions is
  denoted $\rmLip^\alpha_{\rmloc}(X; E)$. When $E = \R$ we abbreviate
  $\rmLip^\alpha(X)$ or $\rmLip^\alpha_{\rmloc}(X)$.

  For a compact subset $K \subset \R^d$ and $\varepsilon > 0$, we
  define the tubular closed neighborhood $K_\varepsilon = \{ x \in
  \R^d : \operatorname{dist}(x, K) \leq \varepsilon \}$. The reader
  will encounter quite frequently the expression $K_1$, which is a
  specific case of this notation.
  
  We will work within the setting of Federer-Fleming's currents. For
  an in-depth exploration of this subject, we refer the reader to
  \cite{Fede}. In this preliminary part, our focus will be on defining
  common notations, highlighting those that deviate from \cite{Fede},
  and revisiting a few definitions.

  The spaces of $m$-vectors and $m$-covectors are $\bigwedge_m \R^d$
  and $\bigwedge^m \R^d$, and they are respectively given the mass and
  comass norm described in \cite[1.8.1]{Fede}. The bracket notation
  $\langle \cdot, \cdot \rangle$ will be reserved for the duality
  between $m$-covectors and $m$-vectors. The canonical basis of $\R^d$
  is $\boldsymbol{e}_1, \dots, \boldsymbol{e}_d$. We write $\Lambda(d,
  m)$ the set of strictly increasing maps $\{1, \dots, d\} \to \{1,
  \dots, m\}$. We introduce the dual bases of $\bigwedge_m \R^d$ and
  $\bigwedge^m \R^d$
  \begin{align*}
    \boldsymbol{e}_I & = \boldsymbol{e}_{i_1} \wedge \cdots \wedge
    \boldsymbol{e}_{i_m} \\ \diff x_I & = \diff x_{i_1} \wedge \cdots
    \wedge \diff x_{i_m}
  \end{align*}
  for $I \in \Lambda(d, m)$.
  
  All our currents will be defined on $\R^d$ and have typically
  dimension $m$, that is, they will belong to $\calD_m(\R^d)$, the
  topological dual of the space $\calD^m(\R^d)$ of compactly supported
  smooth $m$-forms, see \cite[4.1.7]{Fede}. The boundary, the mass and
  the normal mass of $T \in \calD_m(\R^d)$ are $\partial T$, $\bM(T)$
  and $\bN(T)$. If $\omega$ is a smooth $k$-form and $k \leq m$, the
  current $T \hel \omega \in \calD_{m-k}(\R^d)$ is defined by $T \hel
  \omega(\eta) = T(\omega \wedge \eta)$ for $\eta \in
  \calD^{m-k}(\R^d)$. If $\xi \colon \R^d \to \bigwedge_m \R^d$ is a
  locally integrable $m$-vectorfield, $\calL^d \wedge \xi$ is the
  $m$-current that sends $\omega \in \calD^m(\R^d)$ to $\int_{\R^d}
  \langle \omega, \xi \rangle$. Whenever $E \subset \R^d$ is
  measurable, we denote $\llbracket E \rrbracket \in \calD_d(\R^d)$
  the zero-codimensional current defined by $\llbracket E
  \rrbracket(\omega) = \int_E \langle \boldsymbol{e}_1 \wedge \cdots
  \wedge \boldsymbol{e}_d, \omega \rangle$. If $x \in \R^d$, the
  $0$-current $\llbracket x \rrbracket \in \calD_0(\R^d)$ is defined
  by $\llbracket x \rrbracket(\omega) = \omega(x)$.

  We recall that normal currents and flat chains, as defined in
  \cite{Fede} have compact supports. As such, we may at times evaluate
  such currents against smooth forms that are not compactly supported,
  with no warning.

  The spaces of normal $m$-currents and flat $m$-chains are denoted
  $\bN_m(\R^d)$ and $\bF_m(\R^d)$, following customary notation. For
  any subset $X \subset \R^d$, we write
  \begin{align*}
    \bN_m(X) & = \{T \in \bN_m(\R^d) : \rmspt T \subset X \} \\
    \bF_m(X) & = \{T \in \bF_m(\R^d) : \rmspt T \subset X \}
  \end{align*}
  We define the flat norm of a normal current $T \in \bN_m(\R^d)$ in a
  way which departs from Federer's exposure:
  \begin{align*}
    \bF(T) & = \sup \left\{ T(\omega) : \omega \in \calD^m(\R^d)
    \text{ and } \max \left\{ \| \omega \|_\infty, \| \diff \omega
    \|_\infty \right\} \leq 1 \right\} \\ & = \inf \left\{ \bM(S) +
    \bM(T - \partial S) : S \in \bN_{m+1}(\R^d) \right\}.
  \end{align*}
  The proof of the above equality is similar to
  \cite[4.1.12]{Fede}. From the first equality, it is clear that $\bN$
  and $\bM$ are lower semicontinuous with respect to $\bF$. Note that,
  if $T$ is supported in a compact set $K$, the flat norm we just
  defined may differ from
  \begin{align*}
    \bF_K(T) & = \sup \left\{ T(\omega) : \omega \in \calD^m(\R^d)
    \text{ and } \max \left\{ \| \omega \|_{\infty, K}, \| \diff
    \omega \|_{\infty, K} \right\} \leq 1 \right\} \\ & = \inf \left\{
    \bM(S) + \bM(T - \partial S) : S \in \bN_{m+1}(K) \right\}.
  \end{align*}
  However, if $K$ is a $1$-Lipschitz retract of $\R^d$, it is clear
  that $\bF(T) = \bF_K(T)$. Moreover, this assumption implies that
  $\bF_m(K)$ is the $\bF$-closure of $\bN_m(K)$ within
  $\calD_m(\R^d)$.

  Finally, the letter $C$ will refer generally to a constant, that may
  vary from line to line.
\end{Empty}

The construction of charges ultimately relies on the Federer-Fleming's
compactness theorem of normal currents in flat norm. The following
version uses the flat norm $\bF$ (rather than $\bF_K$ as in
\cite[4.2.17(1)]{Fede}). It can be easily deduced from the original
version. Alternatively, it is possible to reproduce the arguments in
the proof of Federer-Fleming, as was done in
\cite[Theorem~4.2]{DePaMoonPfef}.
  
\begin{Theorem}[Compactness]
  Let $K \subset \R^d$ be compact. For all $c \geq 0$, the ball $\{ T
  \in \bN_m(K) : \bN(T) \leq c\}$ is $\bF$-compact.
\end{Theorem}

\begin{Empty}[Convolution of currents]
  In this article, convolutions will play an important role in
  regularizing charges. First, we need to recall how convolution works
  at the level of currents.  The convolution of a current $T \in
  \calD_m(\R^d)$ with a function $\phi \in C^\infty_c(\R^d)$ is
  defined by
  \[
  (T * \phi)(\omega) = T(\check{\phi} * \omega) \qquad \text{for all }
  \omega \in \calD^m(\R^d)
  \]
  where $\check{\phi}(x) = \phi(-x)$ for all $x \in \R^d$. It is clear
  that $T * \phi \in \calD_m(\R^d)$.

  Throughout the article, we fix a $C^\infty$ function $\Phi \colon
  \R^d \to \R$ with compact support in the closed unit ball of $\R^d$,
  that is nonnegative and such that $\int_{\R^d} \Phi = 1$. For the
  sake of simplicity, we additionally assume that $\Phi$ is even. For
  all $\varepsilon > 0$, we define $\Phi_\varepsilon(x) =
  \varepsilon^{-d} \Phi(\varepsilon^{-1}x)$. Below we compile several
  useful facts concerning $T * \Phi_\varepsilon$, when $T \in
  \bN_m(\R^d)$.
\end{Empty}

\begin{Proposition}
  \label{prop:convT}
  There is $C \geq 0$ such that, for all $T \in \bN_m(\R^d)$ and
  $\varepsilon > 0$,
  \begin{itemize}
    \item[(A)] $\rmspt (T * \Phi_\varepsilon) \subset (\rmspt
      T)_\varepsilon$;
    \item[(B)] $\bM(T * \Phi_\varepsilon) \leq \bM(T)$;
    \item[(C)] $\partial T * \Phi_\varepsilon = \partial (T *
      \Phi_\varepsilon)$, $\bN(T * \Phi_\varepsilon) \leq \bN(T)$ and
      $\bF(T * \Phi_\varepsilon) \leq \bF(T)$;
    \item[(D)] $\bM(T * \Phi_\varepsilon) \leq C
      \varepsilon^{-1} \bF(T)$ and $\bN(T * \Phi_\varepsilon) \leq C
      \varepsilon^{-1} \bF(T)$ if $\varepsilon \leq 1$;
    \item[(E)] $\bF(T - T * \Phi_\varepsilon) \leq \varepsilon
      \bN(T)$.
  \end{itemize}
\end{Proposition}

\begin{proof}
  (A) is immediate.
  
  (B). This is because $\|\omega * \Phi_\varepsilon\|_\infty \leq
  \|\omega\|_\infty$ for all $\omega \in \calD^m(\R^d)$.

  (C). The first part comes from the identity $\diff(\omega *
  \Phi_\varepsilon) = \diff \omega * \Phi_\varepsilon$, valid for any
  $\omega \in \calD^m(\R^d)$. Hence $\bM(\partial(T *
  \Phi_\varepsilon)) = \bM( \partial T * \Phi_\varepsilon) \leq
  \bM(\partial T)$. This easily implies that $\bN(T *
  \Phi_\varepsilon) \leq \bN(T)$ and $\bF(T * \Phi_\varepsilon) \leq
  \bF(T)$.

  (D). Let $S \in \bN_{m+1}(\R^d)$. We apply the identity
  \[
  \partial S = - \sum_{k=1}^d \frac{\partial S}{\partial x_k} \hel
  \diff x_k
  \]
  to the current $S * \Phi_\varepsilon$, which yields
  \[
  \partial(S * \Phi_\varepsilon) = - \sum_{k=1}^d \left(
  S *\frac{\partial \Phi_\varepsilon}{\partial x_k} \right) \hel
  \diff x_k
  \]
  From this, we obtain $\bM(\partial S * \Phi_\varepsilon) \leq C
  \varepsilon^{-1} \bM(T)$.

  Thereafter, we decompose $T = (T - \partial S) + \partial
  S$. Applying (B) and the inequality from the preceding paragraph,
  \[
  \bM(T * \Phi_\varepsilon) \leq \bM \left( (T - \partial S) *
  \Phi_\varepsilon \right) + \bM(\partial S * \Phi_\varepsilon) \leq
  \frac{C}{\varepsilon} \left( \bM(T - \partial S) + \bM(S) \right).
  \]
  Taking the infimum on the right-hand side, as $S$ ranges over
  $\bN_{m+1}(\R^d)$, we obtain the first result.  Then, 
  \[
  \bN(T * \Phi_\varepsilon) = \bM(T * \Phi_\varepsilon) + \bM(\partial
  T * \Phi_\varepsilon) \leq \frac{C}{\varepsilon} \bF(T).
  \]
  
  (E).  For any $z \in \R^d$, let $\tau_z \colon \R^d \to \R^d$ be the
  translation by $z$, and $\tau_{z \#} T$ the pushforward current
  \cite[4.1.7]{Fede}. Let $\omega \in \calD^m(\R^d)$. By
  \cite[4.1.18]{Fede},
  \begin{align*}
    (T - T * \Phi_\varepsilon)(\omega) & = T(\omega) - \int_{\R^d}
    \Phi_\varepsilon(z) (\tau_{z \#} T)(\omega) \diff z \\ & =
    \int_{\R^d} \Phi_\varepsilon(z) \left( T - \tau_{z\#}T
    \right)(\omega) \diff z \\ & \leq \int_{\R^d} \Phi_\varepsilon(z)
    \bF(T - \tau_{z\#}T) \max \{ \|\omega\|_\infty, \|\diff \omega\|_\infty \}
    \diff z \\ & \leq \int_{\R^d} \Phi_{\varepsilon}(z) \lvert z
    \rvert \bN(T) \max \{ \|\omega\|_\infty, \| \diff \omega \|_\infty \} \diff z
    \\ & \leq \varepsilon \bN(T) \max \{ \| \omega \|_\infty, \|\diff
    \omega\|_\infty \}. \qedhere
  \end{align*}
\end{proof}

\section{Charges in middle dimension}
\label{sec:charges}

\begin{Empty}[Charges in middle dimension]
  \label{e:charges}
  In this section, we provide a self-contained introduction to charges
  in middle dimension and establish their fundamental properties. This
  notion was pioneered by De Pauw, Moonens and Pfeffer
  in~\cite{DePaMoonPfef}.

  An \emph{$m$-charge} over a compact set $K \subset \R^d$ is a linear
  map $\omega \colon \bN_m(K) \to \R$ that satisfies one of the
  following equivalent continuity properties:
  \begin{enumerate}
  \item[(A)] $\omega(T_n) \to 0$ for any bounded sequence $(T_n)$ in
    $\bN_m(K)$ that converges in flat norm to $0$;
  \item[(B)] the restriction of $\omega$ to the unit ball of
    $\bN_m(K)$ is $\bF$-continuous;
  \item[(C)] for all $\varepsilon > 0$, there is some $\theta \geq 0$
    such that
    \[
    |\omega(T)| \leq \varepsilon \bN(T) + \theta \bF(T)
    \]
    holds for any normal current $T \in \bN_m(K)$.
  \end{enumerate}
  One clearly has (A) $\iff$ (B) and (C) $\implies$ (A). The only non
  trivial implication (A) $\implies$ (C) can be derived as a short
  consequence of the compactness theorem. Indeed, suppose by
  contradiction that (A) holds and (C) is false. In this case, there
  is $\varepsilon > 0$ and a sequence $(T_n)$ of normal currents
  supported in $K$, with normal masses $\bN(T_n) = 1$ such that
  \begin{equation}
    \label{eq:1}
    |\omega(T_n)| > n \bF(T_n) + \varepsilon
  \end{equation}
  for all $n$. Some subsequence $(T_{n_k})$ converges to a normal
  current $T \in \bN_m(K)$ in flat norm. Property (A) then implies
  that $\omega(T_{n_k}) \to \omega(T)$. Consequently, $\bF(T_{n_k})
  \leq n_k^{-1} |\omega(T_{n_k})|$ tends to $0$ as $k \to \infty$,
  which implies that $T = 0$ and $\omega(T_{n_k}) \to 0$. This is in
  contradiction with~\eqref{eq:1}.
\end{Empty}

\begin{Empty}
  The space of $m$-charges over $K$ is denoted $\bCH^m(K)$. As $\bF
  \leq \bN$, the continuity property (C) above implies that charges
  are $\bN$-continuous, \textit{i.e} $\bCH^m(K)$ is a subspace of the
  dual $\bN_m(K)^*$. We set
  \[
  \| \omega \|_{\bCH^m(K)} \defeq \sup \left\{ \omega(T) : T \in
  \bN_m(K) \text{ and } \bN(T) \leq 1 \right\}.
  \]
  In fact, $\bCH^m(K)$ is a closed subspace of $\bN_m(K)^*$, for if
  $(\omega_n)$ is a sequence in $\bCH^m(K)$ converging towards $\omega
  \in \bN_m(K)^*$, then $\omega_n \to \omega$ uniformly on the unit
  ball of $\bN_m(K)$. We conclude by (B) that $\omega$ is a charge.

  In addition, there is also a notion of \emph{weak convergence of
    charges}: we say that $\omega_n \to \omega$ weakly whenever
  $\omega_n(T) \to \omega(T)$ for all $T \in \bN_m(K)$.
\end{Empty}

\begin{Empty}[Exterior derivative]
  Operations on normal currents, such as pushforwards by Lipschitz
  maps, taking the boundary, have a counterpart in term of charges,
  defined by duality. Here we focus only on defining the
  \emph{exterior derivative} $\diff \omega \in \bCH^{m+1}(K)$, by
  setting
  \[
  \diff \omega(T) \defeq \omega(\partial T)
  \]
  for all $T \in \bN_m(K)$. That $\diff \omega$ is continuous, in the
  sense of charges, is a consequence of the identities
  \begin{equation}
    \label{eq:FNboundary}
  \bN(\partial T) \leq \bN(T), \qquad \bF(\partial T) \leq \bF(T)
  \end{equation}
  that furthermore implies that $\mathrm{d} \colon \bCH^{m}(K) \to
  \bCH^{m+1}(K)$ is bounded. As $\partial \circ \partial = 0$, we have
  $\diff \circ \diff = 0$. In other words, $(\bCH^{\bullet}(K),
  \diff)$ is a cochain complex. In fact, the De Pauw-Moonens-Pfeffer
  representation theorem expresses that this is the smallest cochain
  complex spanned by continuous differential forms.

  Before we state this result properly, we will first identify
  $0$-charges with continuous functions. To each $0$-charge $\omega$,
  we associate the function $\Gamma(\omega) \in C(K)$ defined by
  $\Gamma(\omega)(x) = \omega(\llbracket x \rrbracket)$. The
  continuity of $\Gamma(\omega)$ is a consequence of that of $\omega$.
\end{Empty}

\begin{Theorem}
  \label{thm:CH0C}
  $\Gamma \colon \bCH^0(K) \to C(K)$ is a Banach space isomorphism.
\end{Theorem}

\begin{proof}
  First we check that $\Gamma$ is a continuous. For all $x \in K$, we
  have
  \[
  |\Gamma(\omega)(x)| \leq \|\omega\|_{\bCH^0(K)} \bM(\llbracket x
  \rrbracket) = \|\omega\|_{\bCH^0(K)}.
  \]
  Thus $\|\Gamma(\omega)\|_\infty \leq \|\omega\|_{\bCH^0(K)}$.

  Next we claim that $\Gamma$ is injective. Let us call $\bP_0(K)$ the
  space of polyhedral $0$-currents supported in $K$, \textit{i.e} the
  linear subspace of $\calD_0(\R^d)$ spanned by the $\llbracket x
  \rrbracket$, $x \in K$. By an easy corollary of the deformation
  theorem~\cite[4.2.9]{Fede}, every $T \in \bN_0(K)$ is the
  $\bF$-limit of a sequence $(T_n)$ in $\bP_0(K)$ such that $\bM(T_n)
  \uparrow \bM(T)$.

  Let $\omega \in \ker \Gamma$, so that $\omega$ vanishes on
  $\bP_0(K)$. By the preceding result and the continuity property of
  charges, $\omega = 0$. This proves that $\Gamma$ is injective.

  Next we prove the surjectivity of $\Gamma$. Let $g \in C(K)$. We
  define the function $\omega$, on polyhedral $0$-chains, by
  \[
  \omega \left( \sum_{k=1}^n a_k \llbracket x_k \rrbracket
  \right) = \sum_{k=1}^n a_k g(x_k).
  \]
  Let $\varepsilon > 0$. There is a Lipschitz function $f \in
  \rmLip(K)$ such that $\|f - g\|_\infty \leq \varepsilon$. Setting
  $\theta = \max\{\|f\|_\infty, \rmLip f\}$, we have
  \begin{align}
    |\omega(T)| & \leq \left| \sum_{k=1}^n a_k f(x_k) \right|+ \left|
    \sum_{k=1}^n a_k (f - g)(x_k) \right| \notag \\ & \leq \theta
    \bF(T) + \varepsilon \bM(T) \label{eq:2}
  \end{align}
  for every $0$-polyhedral chain $T$. We extend $\omega$ to $\bM_0(K)$
  with
  \[
  \omega(T) = \lim_{n \to \infty} \omega(T_n)
  \]
  where $T \in \bM_0(K)$ and $(T_n)$ is any sequence of polyhedral
  $0$-chains that $\bF$-converges to $T$, with $\bM(T_n) \uparrow
  \bM(T)$. By~\eqref{eq:2}, the quantity $\omega(T)$, thus defined,
  does not depend on the approximating sequence. It is also
  straightforward that $\omega$ is linear and~\eqref{eq:2} holds now
  for any $T \in \bM_0(K)$. Hence $\omega \in \bCH^0(K)$ and
  $\Gamma(\omega) = g$. This proves that $\Gamma$ is onto.

  Finally, $\Gamma^{-1}$ is continuous by the open mapping theorem.
\end{proof}

\begin{Empty}[Continuous differential forms]
  A continuous $m$-form $\omega \in C(K, \bigwedge^m \R^d)$ act on
  a normal current $T \in \bN_m(K)$ by means of
  \[
  \omega(T) \defeq \int_{\rmspt T} \langle \omega(x), \vec{T}(x)
  \rangle \diff \|T\|(x)
  \]
  We argue that this formula makes $\omega$ into an $m$-charge (we are
  committing a slight abuse of notation by using the same symbol
  $\omega$ for both the continuous form and the corresponding charge,
  even though the obvious mapping $C(K, \bigwedge^m \R^d) \to
  \bCH^m(K)$ may be not injective). Linearity is clear. As for
  continuity, let us fix $\varepsilon > 0$ and choose a compactly
  supported smooth form $\phi \colon \R^d \to \bigwedge^m \R^d$ such
  that $|\omega(x) - \phi(x)| \leq \varepsilon$ for all $x \in
  K$. Then, for all $T \in \bN_m(K)$,
  \begin{align*}
    |\omega(T)| & \leq \left| \int_{\rmspt T} \langle \omega(x) -
    \phi(x), \vec{T}(x) \rangle \diff \|T\|(x)\right| + |T(\phi)| \\ &
    \leq \varepsilon \bM(T) + \max \{ \|\phi\|_\infty, \| \diff \phi
    \|_\infty \} \bF(T) \\
    & \leq \varepsilon \bN(T) + \theta \bF(T)
  \end{align*}
  for $\theta = \max \{ \|\phi\|_\infty, \| \diff \phi\|_\infty\}$.
\end{Empty}

We now state a criterion for relative compactness in $\bCH^m(K)$. It
will prove useful in the next section for establishing the basic
properties of the space of fractional charges.

\begin{Theorem}
  \label{thm:compact}
  Let $\Omega \subset \bCH^m(K)$. The following are equivalent:
  \begin{itemize}
  \item[(A)] $\Omega$ is relatively compact;
  \item[(B)] the continuity inequality
    \[
    |\omega(T)| \leq \varepsilon \bN(T) + \theta \bF(T)
    \]
    holds for all $\omega \in \Omega$ and $T \in \bN_m(K)$, with a
    $\theta = \theta(\varepsilon) \geq 0$ that can chosen
    independently of $\omega$.
  \end{itemize}
\end{Theorem}

\begin{proof}
  (A) $\implies$ (B). We prove this implication by
  contradiction. Suppose there are $\varepsilon > 0$ and two sequences
  $(\omega_n)$ in $\Omega$ and $(T_n)$ in $\bN_m(K)$ such that
  \[
  |\omega_n(T_n)| > \varepsilon \bN(T_n) + n \bF(T_n) 
  \]
  for all integers $n$. We can also suppose $\bN(T_n) = 1$ for all
  $n$. As $\Omega$ is relatively compact, it is bounded, consequently
  \[
  n \bF(T_n) < \sup_{\omega \in \Omega} \|\omega\|_{\bCH^m(K)} <
  \infty
  \]
  which implies that $(T_n)$ converges to $0$ in flat norm. On the
  other side, there is a subsequence $(\omega_{n_k})$ that converges
  to $\omega \in \bCH^m(K)$. Hence
  \[
  |\omega_{n_k}(T_{n_k})| \leq |\omega(T_{n_k})| + \|\omega -
  \omega_{n_k}\|_{\bCH^m(K)} \to 0
  \]
  which contradicts that $|\omega_{n_k}(T_{n_k})| > \varepsilon$.

  (B) $\implies$ (A). Denote by $B_{\bN_m(K)}$ the unit ball of
  $\bN_m(K)$, metrized by $\bF$, and let $\iota \colon \bCH^m(K) \to
  C(B_{\bN_m(K)})$ be the linear map that sends a charge to its
  restriction to $B_{\bN_m(K)}$. Here, $C(B_{\bN_m(K)})$ is given the
  supremum norm, so that $\iota$ is an isometric embedding.

  Since $\bCH^m(K)$ is a Banach space, $\iota(\bCH^m(K))$ is
  closed. We only need to show that $\iota(\Omega)$ is relatively
  compact in $C(B_{\bN_m(K)})$.

  First, the inequality in (B) (for $\varepsilon = 1$) entails that
  $\iota(\Omega)$ is pointwise bounded. Now, for an arbitrary
  $\varepsilon > 0$ there is $\theta \geq 0$ as in (B). If $T, S \in
  B_{\bN_m(K)}$ satisfy $\bF(T - S) \leq \varepsilon / \theta$, then
  for any $\omega \in \Omega$, one has
  \[
  |\iota(\omega)(T) - \iota(\omega)(S)| \leq \theta \bF(T - S) +
  \varepsilon \bN(T - S) \leq 3 \varepsilon.
  \]
  This proves that $\iota(\Omega)$ is equicontinuous, thus relatively
  compact by the Arzelà-Ascoli theorem. The proof is then finished.
\end{proof}

\begin{Empty}[Charges over $\R^d$]
  \label{e:chargesoverRd}
  It is possible, as was done in~\cite{DePaMoonPfef}, to extend the
  notion of $m$-charge over arbitrary subsets of $\R^d$. In this
  article, we will not attempt to be as general as possible, but
  rather concentrate on charges defined over $\R^d$, a domain
  particularly suited for performing convolutions or introducing the
  so-called Littlewood-Paley decomposition of Section~\ref{sec:LPdec}.

  We call $m$-charge over $\R^d$ a linear functional $\omega \colon
  \bN_m(\R^d) \to \R$ whose restriction to $\bN_m(K)$ is an element of
  $\bCH^m(K)$, for all compact subsets $K \subset \R^d$. The space of
  $m$-charges is denoted $\bCH^m(\R^d)$, and we equip it with the
  Fréchet topology induced by the family of seminorms
  \[
  \| \omega \|_{\bCH^m(K)} = \sup \left\{ \omega(T) : T \in \bN_m(K)
  \text{ and } \bN(T) \leq 1 \right\}
  \]
  where $K$ ranges over all compact subsets of $\R^d$.

  The results proven in this section carry over to
  $\R^d$. Specifically,
  \begin{itemize}
  \item[(A)] the definition of weak convergence, that of the exterior
    derivative in $\bCH^m(\R^d)$ is unchanged.
  \item[(B)] The map $\Gamma$ that sends $\omega \in \bCH^m(\R^d)$ to
    the continuous function $x \mapsto \omega(\llbracket x
    \rrbracket)$ is a Fréchet space isomorphism. Here, the topology
    $C(\R^d)$ is induced by the seminorms $\| \cdot \|_{\infty, K}$,
    where $K \subset \R^d$ is compact.
  \item[(C)] Any continuous form $\omega \in C(\R^d; \bigwedge^m
    \R^d)$ can be regarded as a charge.
  \item[(D)] A subset $\Omega \subset \bCH^m(\R^d)$ is relatively
    compact if and only if, for every compact set $K \subset \R^d$ and
    every $\varepsilon > 0$, there is $\theta = \theta(K,\varepsilon)$
    such that for all $T \in \bN_m(K)$, we have $|\omega(T)| \leq
    \varepsilon \bN(T) + \theta \bF(T)$.
  \end{itemize}
\end{Empty}

\begin{Empty}[Regularization by convolution]
  Let $\phi \in C^\infty_c(\R^d)$. We define the \emph{convolution} of
  a charge $\omega \in \bCH^m(\R^d)$ with $\phi$ by
  \[
  (\omega * \phi)(T) = \omega(T * \check{\phi}) \qquad \text{for all }
  T \in \bN_m(\R^d).
  \]
  The next proposition shows that this construction yields a smooth
  form.
\end{Empty}

\begin{Proposition}
  \label{prop:smooth}
  Let $\omega \in \bCH^m(\R^d)$ and $\phi \in C^\infty_c(\R^d)$. Then
  $\omega * \phi \in C^\infty(\R^d; \bigwedge^m \R^d)$. Explicitly,
  \[
  (\omega * \phi)(z) = \sum_{I \in \Lambda(d, m)} \omega \left(
  \calL^d \wedge \phi( z - \cdot ) \boldsymbol{e}_I \right) \diff x_I
  \text{ for all } z \in \R^d.
  \]
\end{Proposition}

\begin{proof}
  Call $\tilde{\omega}(z)$ the right-hand side. First we check that
  $\tilde{\omega}$ is a smooth $m$-form. This is done by ensuring
  that, for all $1 \leq i \leq d$ and for any sequence $(h_n)$ of
  nonzero real numbers tending to 0,
  \[
  \frac{ \calL^d \wedge \phi( z + h_n \boldsymbol{e}_i - \cdot )
    \boldsymbol{e}_I - \calL^d \wedge \phi( z - \cdot )
    \boldsymbol{e}_I }{h_n} \to \calL^d \wedge \frac{\partial
    \phi}{\partial x_i}(z - \cdot) \boldsymbol{e}_I
  \]
  in flat norm with uniformly bounded normal masses.

  Next, in order to prove that the charges $\omega * \phi$ and
  $\tilde{\omega}$ coincide, we need only do so on currents of the
  form $\calL^d \wedge \xi$, where $\xi = \sum_{I \in \Lambda(d, m)}
  \xi_I \boldsymbol{e}_I$ is a compactly supported smooth
  $m$-vectorfield. This is because, for all $T \in \bN_m(\R^d)$,
  \[
  (\omega * \phi - \tilde{\omega})(T) = \lim_{\varepsilon \to 0}
  (\omega * \phi - \tilde{\omega})(T * \Phi_\varepsilon)
  \]
  by Proposition~\ref{prop:convT}(C) and (E), and $T *
  \Phi_\varepsilon$ has the form $\calL^d \wedge \xi$ by
  \cite[4.1.2]{Fede}.

  We begin by evaluating
  \[
  (\calL^d \wedge \xi)(\tilde{\omega}) = \sum_{I \in \Lambda(d,m)}
  \int_{\R^d} \omega(\calL^d \wedge \phi( z - \cdot )
  \boldsymbol{e}_I) \xi_I(z) \diff z.
  \]
  On the other hand, one has
  \[
  \omega(T * \check{\phi}) = \sum_{I \in \Lambda(d,m)} \omega \left(
  \calL^d \wedge \xi_I * \check{\phi} \boldsymbol{e}_I \right).
  \]
  Following~\cite[4.1.2]{Fede}, we introduce, for every $n \geq 1$, a
  partition $A_{n,1}, \dots, A_{n,p_n}$ of $\rmspt \xi$ into Borel
  sets of diameter less than $n^{-1}$ and choose points $z_{n, k} \in
  A_{n, k}$ for $1 \leq k \leq p_n$. Then
  \[
  \sum_{k=1}^{p_n} \xi_I(z_{n,k}) \left(\calL^d \wedge \phi(z_{n,k} - \cdot)
  \boldsymbol{e}_I\right) \calL^d(A_{n,k}) \to \calL^d \wedge \xi_I *
  \check{\phi} \boldsymbol{e}_I
  \]
  in flat norm with uniformly bounded normal masses. Thus,
  \begin{align*}
    \omega(T * \check{\phi}) & = \lim_{n \to \infty} \sum_{I \in
      \Lambda(d,m)} \sum_{k=1}^{p_n} \xi_I(z_{n,k}) \omega(\calL^d \wedge
    \phi(z_{n,k} - \cdot) \boldsymbol{e_I}) \calL^d(A_{n,k}) \\ & =
    \sum_{I \in \Lambda(d,m)} \int_{\R^d} \xi_I(z) \omega(\calL^d
    \wedge \phi(z - \cdot) \boldsymbol{e_I}) \diff z \\ & = (\calL^d
    \wedge \xi)(\tilde{\omega}). \qedhere
  \end{align*}
\end{proof}

\section{$\alpha$-Fractionality}

\begin{Empty}
  Let $\alpha \in \mathopen{(} 0, 1 \mathclose{]}$.  An
    \emph{$\alpha$-fractional $m$-charge} over a compact set $K
    \subset \R^d$ is a linear functional $\omega \colon \bN_m(K) \to
    \R$ for which there is a constant $C \geq 0$ such that
    \[
    |\omega(T)| \leq C \bN(T)^{1-\alpha} \bF(T)^\alpha \text{ for all } T \in \bN_m(K).
    \]
  It is clear that the above requirement is a stronger condition than
  the charge continuity property stated in Paragraph~\ref{e:charges}.

  We adopt the notation $\bCH^{m, \alpha}(K)$ to represent the space
  of $\alpha$-fractional $m$-charges, normed by
  \begin{equation}
    \label{eq:seminormCHalpha}
  \|\omega\|_{\bCH^{m, \alpha}(K)} = \inf \left\{ C \geq 0 :
  |\omega(T)| \leq C \bN(T)^{1-\alpha} \bF(T)^{\alpha} \text{ for all
  } T \in \bN_m(K) \right\}.
  \end{equation}
  We also define $\| \omega \|_{\bCH^{m, \alpha}(K)} = \infty$ if
  $\omega \in \bCH^m(K) \setminus \bCH^{m, \alpha}(K)$.
  
  The parameter $\alpha$ represents regularity. Equivalently, an
  $m$-charge $\omega$ is $\alpha$-fractional whenever its restriction
  to the unit ball of $\bN_m(K)$, endowed with the distance inherited
  from $\bF$, is $\alpha$-Hölder continuous. One clearly has
  inclusions
  \[
  \bCH^{m,\beta}(K) \subset \bCH^{m, \alpha}(K) \subset \bCH^m(K)
  \]
  (that are continuous) whenever $\beta \geq \alpha$. In addition, the
  reader may use the continuity of the second embedding and the lower
  semicontinuity of $\|\cdot\|_{\bCH^{m, \alpha}(K)}$ with respect to
  weak convergence to check that $\bCH^{m, \alpha}(K)$ is a Banach
  space.

  When $\alpha = 1$ and $K$ is a Lipschitz neighborhood retract, we
  encounter a well-known object. Indeed, in this case, a
  $1$-fractional charge $\omega$ is $\bF$-continuous and $\bN_m(K)$ is
  $\bF$-dense in $\bF_m(K)$. As such, $\omega$ can be uniquely
  extended so as to become an element of $\bF_m(K)^*$, the space of
  \emph{flat $m$-cochains over $K$}, introduced by H. Whitney.

  More generally, we can think of $\alpha$-fractionality as a
  regularity that is intermediate between that of mere charges and
  that of flat cochains. We observe that, as a consequence
  of~\eqref{eq:FNboundary}, the exterior derivative of an
  $\alpha$-fractional $m$-charge is again $\alpha$-fractional (and the
  map $\diff \colon \bCH^{m, \alpha}(K) \to \bCH^{m+1, \alpha}(K)$ is
  continuous).

  We can of course define $\alpha$-fractional $m$-charges over the
  whole space $\R^d$. They are by definition the charges $\omega \in
  \bCH^m(\R^d)$ whose restrictions to $\bN_m(K)$ belong to $\bCH^{m,
    \alpha}(K)$, for each compact subset $K$ of $\R^d$. The space they
  form is denoted $\bCH^{m, \alpha}(\R^d)$, and we give it the locally
  convex topology induced by the seminorms $\| \cdot \|_{\bCH^{m,
      \alpha}(K)}$ defined as in~\eqref{eq:seminormCHalpha}, where $K$
  ranges over all compact sets. The elements of $\bCH^{m,1}(\R^d)$
  correspond to the locally flat $m$-cochains over $\R^d$ described in
  \cite[Section~4]{FedeReal}.

  It may seem strange that the coefficient quantifying the regularity
  of a fractional charge is not diminished by $1$ when the exterior
  derivative is applied. This is why we introduced the term
  $\alpha$-fractionality. It is already the case that the exterior
  derivative of a flat cochain remains a flat cochain. The distinction
  between working with generalized differential forms $\omega$ and,
  say, Sobolev functions, whose high-order distributional derivatives
  are increasingly poorly controlled, is that $\diff^k \omega = 0$ for
  $k \geq 2$.
  
  In the next paragraphs we will exhibit two important examples of
  fractional charges, which served as motivations for the definition.
\end{Empty}

\begin{Empty}[Relationship with Hölder charges]
  Here we look at the zero-codimensional case $m = d$ and $K = [0,
    1]^d$. We say that a measurable set $E \subset [0, 1]^d$ has
  finite perimeter whenever $\llbracket E \rrbracket \in \bN_d(K)$ and
  we denote by $\niceBV(K)$ the algebra of such sets. To each
  $\alpha$-fractional $d$-charge $\omega$, we associate the map
  $\Upsilon(\omega) \colon \niceBV(K) \to \R$ that sends $E$ to
  $\omega(\llbracket E \rrbracket)$. We set
  \[
  \gamma \defeq \frac{d-1}{d} + \frac{\alpha}{d} \in \left(
  \frac{d-1}{d}, 1 \right].
  \]
  The map $\mu = \Upsilon(\omega)$ satisfies the following properties
  \begin{enumerate}
  \item[(A)] Finite additivity: for disjoint sets with finite
    perimeters $E, F \in \niceBV(K)$, we have $\mu(E \cup F) =
    \mu(E) + \mu(F)$;
  \item[(B)] Continuity: if $(E_n)$ is a sequence in $\niceBV(K)$ with
    uniformly bounded perimeters $\sup_n \bM(\partial\llbracket E_n
    \rrbracket) < \infty$ and such that $\calL^d(E_n) \to 0$, we have
    $\mu(E_n) \to 0$;
  \item[(C)] Hölder control over dyadic cubes: there is a constant $C
    \geq 0$ such that $|\mu(Q)| \leq C \calL^d(Q)^\gamma$ for all
    dyadic cubes $Q \subset K$.
  \end{enumerate}
  The last property comes from
  \begin{align*}
  |\omega(\llbracket Q \rrbracket)| & \leq \| \omega \|_{\bCH^{d,
      \alpha}(K)} \bN(\llbracket Q \rrbracket)^{1-\alpha}
  \bF(\llbracket Q \rrbracket)^{\alpha} \\ & \leq (1 + 2d)^{1-\alpha}
  \| \omega \|_{\bCH^{d, \alpha}(K)} \calL^d(Q)^{(1-\alpha)
    \frac{d-1}{d}} \calL^d(Q)^\alpha.
  \end{align*}
  The maps $\mu \colon \niceBV(K) \to \R$ that satisfy (A), (B) and
  (C) appeared in \cite{BouaDePa, Boua} under the name
  \emph{$\gamma$-Hölder charges}. The article \cite{BouaDePa} exhibits
  some examples of Hölder charges derived from stochastic processes,
  whereas \cite{Boua} used Hölder charges as integrators for
  Young-type multidimensional integrals. The space of $\gamma$-Hölder
  charges is designated $\rmsch^\gamma(K)$. We claim that
  \[
  \Upsilon \colon \bCH^{d, \alpha}(K) \to \rmsch^\gamma(K)
  \]
  is a one-to-one correspondence. Note that this paragraph violates
  our promise of a self-contained article, since it relies on material
  from \cite{Boua}. However, the result we demonstrate here, and even
  more so, the notion of (top-dimensional) Hölder charge, will not be
  utilized further in the article.

  First we prove that $\Upsilon$ is injective. We recall that each
  normal current in $\bN_d(K)$ has the form $\llbracket K \rrbracket
  \hel f$, where $f$ is a function with bounded variation supported in
  $K$. We can show the existence of a sequence of functions $(f_n)$
  supported in $K$, each $f_n$ being constant on dyadic cubes of
  side length $2^{-n}$, such that $\sup_n \bN(\llbracket K \rrbracket
  \hel f_n) < \infty$ and $\bF(\llbracket K \rrbracket \hel f -
  \llbracket K \rrbracket \hel f_n) \to 0$, see for example
  \cite[Lemma~4.8]{BouaDePa}. If $\omega$ is an $\alpha$-fractional
  $m$-charge in the kernel of $\Upsilon$, then $\omega(\llbracket K
  \rrbracket \hel f_n) = 0$ for each $n$. By letting $n \to \infty$,
  we obtain that $\omega$ vanishes on a general normal current
  $\llbracket K \rrbracket \hel f$.

  Conversely, for each $\mu \in \rmsch^\gamma(K)$ and each nonnegative
  function $f$ with bounded variation supported in $K$, we set
  \[
  \Upsilon^{-1}(\mu)(\llbracket K \rrbracket \hel f) = \int_0^\infty \mu(\{f > t\}) \diff t.
  \]
  For a general $f$, we define
  \[
  \Upsilon^{-1}(\mu)(\llbracket K \rrbracket \hel f) =
  \Upsilon^{-1}(\mu)(\llbracket K \rrbracket \hel f^+) -
  \Upsilon^{-1}(\mu)(\llbracket K \rrbracket \hel f^-)
  \]
  where $f^+$ and $f^-$ are the positive and negative parts of $f$. We
  claim that $\Upsilon^{-1}(\mu)$ is an $\alpha$-fractional
  $m$-charge. This results essentially comes from \cite[Theorem~3.10]{Boua}, where
  it was proven that
  \[
  \mu(E) \leq C \bM(\partial \llbracket E \rrbracket)^{1-\alpha}
  \calL^d(E)^\alpha
  \]
  for some constant $C = C(\mu)$ and for any $E \in
  \niceBV(K)$. Applying Young's inequality, we obtain
  \[
  \mu(E) \leq C \left( (1-\alpha) \lambda \bM(\partial \llbracket E
  \rrbracket) + \frac{\alpha}{\lambda^{\frac{1}{\alpha} - 1}}
  \calL^d(E) \right)
  \]
  for all $\lambda > 0$. Let $f$ be a nonnegative function with
  bounded variation supported in $K$. By the coarea formula,
  \begin{align*}
  |\Upsilon^{-1}(\mu)(\llbracket K \rrbracket \hel f)| & \leq
  C(1-\alpha) \lambda \int_0^\infty \bM(\partial (\llbracket K
  \rrbracket \hel \{f > t\})) \diff t \\
  & \qquad \qquad \qquad \qquad +
  \frac{C\alpha}{\lambda^{\frac{1}{\alpha} - 1}} \int_0^\infty
  \calL^d\left(\{ f > t \}\right) \diff t \\ & \leq C
  (1-\alpha)\lambda \bN(\llbracket K \rrbracket \hel f) +
  \frac{C\alpha}{\lambda^{\frac{1}{\alpha}-1}} \bF(\llbracket K
  \rrbracket \hel f)
  \end{align*}
  Choosing
  \[
  \lambda = \left( \frac{\bF(\llbracket K \rrbracket \hel
    f)}{\bN(\llbracket K \rrbracket \hel f)}\right)^{\alpha}.
  \]
  we obtain
  \[
  |\Upsilon^{-1}(\mu)(\llbracket K \rrbracket \hel f)| \leq C
  \bN(\llbracket K \rrbracket \hel f)^{1 - \alpha} \bF(\llbracket K
  \rrbracket \hel f)^\alpha.
  \]
  An inequality of the type above is easily obtained as well when we
  remove the restriction on the sign of $f$. This ends the proof that
  $\Upsilon^{-1}(\mu) \in \bCH^{d, \alpha}(K)$. We let the reader
  check that $\Upsilon$ and $\Upsilon^{-1}$ are reciprocal maps.
\end{Empty}

\begin{Empty}[Hölder differential form]
  We let $\rmLip^\alpha_{\rmloc}(\R^d, \bigwedge^m \R^d)$ the space of
  locally $\alpha$-Hölder continuous $m$-forms. It is a Fréchet space,
  when given the family of seminorms
  \[
  \|\omega\|_{\rmLip^\alpha(K, \bigwedge^m \R^d)} \defeq
  \max \left\{\|\omega\|_{\infty, K}, \rmLip^\alpha(\omega; K) \right\}
  \]
  indexed over all compact subsets $K$ of $\R^d$.

  We claim that, for $\omega \in \rmLip^\alpha_{\rmloc}(\R^d,
  \bigwedge^m \R^d)$, the corresponding charge is
  $\alpha$-fractional. We recall that $\Phi \colon \R^d\to \R$ be a
  smooth nonnegative even function, supported in the closed unit ball,
  with $\int_{\R^d} \Phi = 1$. For all $\varepsilon \in \mathopen{(}
  0, 1 \mathclose{]}$, we set $\Phi_\varepsilon(x) =
    \frac{1}{\varepsilon^d} \Phi \left( \frac{x}{\varepsilon} \right)$
    and
  \[
  \omega_\varepsilon (x) = \omega * \Phi_\varepsilon (x) = \int_{\R^d}
  \omega(y) \Phi_\varepsilon(x - y) \diff y.
  \]
  For each $x$ in a compact set $K$, we have
  \[
  \| \omega - \omega_\varepsilon \|_{\infty, K} \leq \rmLip^\alpha(\omega;
  K_1) \varepsilon^\alpha
  \]
  where $K_1 \defeq \{x \in \R^d : \rmdist(x, K) \leq 1 \}$. Moreover,
  $\omega_\varepsilon$ is smooth, and for all $i \in \{1, \dots, d\}$,
  we have
  \begin{align*}
  \partial_i \omega_\varepsilon(x) & = \frac{1}{\varepsilon^{d+1}}
  \int_{\R^d} \omega(y) \partial_i \Phi \left( \frac{x -
    y}{\varepsilon} \right) \diff y \\
  & = \frac{1}{\varepsilon^{d+1}}
  \int_{\R^d} ( \omega(y) - \omega(x) ) \partial_i \Phi \left( \frac{x -
    y}{\varepsilon} \right) \diff y
  \end{align*}
  from which we infer that
  \[
  |\partial_i \omega_\varepsilon(x)| \leq \rmLip^\alpha(\omega;
  K_1) \left( \int_{\R^d} |\partial_i \Phi| \right)
  \frac{1}{\varepsilon^{1 - \alpha}}.
  \]
  and thus,
  \[
  \| \diff \omega_\varepsilon \|_{\infty, K} \leq C
  \rmLip^\alpha(\omega, K_1) \frac{1}{\varepsilon^{1 -
      \alpha}}
  \]
  for some constant depending on $d$.  Next, for $T \in \bN_m(K)$, we
  estimate
  \[
  \omega(T) = \int_{K} \langle \omega(x) - \omega_\varepsilon(x),
  \vec{T}(x) \rangle \diff \|T\| + T(\omega_\varepsilon)
  \]
  From the above inequalities, we can control the first term
  \[
  \left| \int_{K} \langle \omega(x) - \omega_\varepsilon(x),
  \vec{T}(x) \rangle \diff \|T\| \right| \leq \rmLip^\alpha(\omega;
  K_1) \varepsilon^\alpha \bN(T)
  \]
  whereas, if we suppose furthermore that $K$ is a $1$-Lipschitz retract,
  the second term is controlled by
  \begin{align*}
  |T(\omega_\varepsilon)| & \leq \max\{ \| \omega_\varepsilon
  \|_{\infty, K}, \| \diff \omega_\varepsilon \|_{\infty, K} \} \bF(T)
  \\ & \leq \frac{1}{\varepsilon^{1 - \alpha}} \max \{ \| \omega
  \|_{\infty, K_1}, C \rmLip^\alpha(\omega; K_1) \} \bF(T)
  \end{align*}
  By choosing $\varepsilon = \bF(T)^\alpha / \bN(T)^\alpha$ (which is
  indeed less than 1) and combining the preceding inequalities, we
  obtain
  \[
  |\omega(T)| \leq C_K \bN(T)^{1 - \alpha} \bF(T)^\alpha,
  \]
  with $C_K \defeq \rmLip^\alpha(\omega; K_1) + \max \{\|
  \omega\|_{\infty, K_1}, C \rmLip^\alpha (\omega; K_1) \}$.
\end{Empty}

Next we will identify $\alpha$-fractional $0$-charges with (locally)
Hölder continuous functions.

\begin{Proposition}
  \label{prop:CH0alphaHolder}
  The map $\Gamma$ restricts to a Banach space isomorphism between
  $\bCH^{0, \alpha}(K)$ and $\rmLip^\alpha(K)$.
\end{Proposition}

\begin{proof}
  If $\omega$ is $\alpha$-fractional, then
  \begin{align*}
  |\Gamma(\omega)(x) - \Gamma(\omega)(y)| & \leq
  \|\omega\|_{\bCH^{0,\alpha}(K)} \bF(\llbracket x \rrbracket -
  \llbracket y \rrbracket)^\alpha \bM(\llbracket x \rrbracket -
  \llbracket y \rrbracket)^{1 - \alpha} \\ & \leq 2^{1 - \alpha}
  \|\omega\|_{\bCH^{0,\alpha}(K)} \lvert x - y \rvert^\alpha
  \end{align*}
  for all $x, y \in K$. Furthermore, $\|\Gamma(\omega)\|_\infty \leq
  \|\omega\|_{\bCH^0(K)} \leq \|\omega\|_{\bCH^{0,\alpha}(K)}$. This
  also shows that $\Gamma \colon \bCH^{0,\alpha}(K) \to
  \rmLip^\alpha(K)$ is continuous.
  
  Conversely, let $f \in \rmLip^\alpha(K)$ and $\omega =
  \Gamma^{-1}(f)$ the associated $0$-charge. Let $\varepsilon >
  0$. The function
  \[
  f_\varepsilon \colon x \in K \mapsto \min \left\{ f(y) +
  \frac{\rmLip^\alpha f}{\varepsilon^{1 - \alpha}} d(x, y) : y \in K
  \right\}
  \]
  belongs to $\rmLip(K)$, and $\rmLip f_\varepsilon \leq
  (\rmLip^\alpha f) \varepsilon^{\alpha-1}$. It is also clear that
  $f_\varepsilon(x) \leq f(x)$ for all $x \in K$. If $d(x, y) \geq
  \varepsilon$, then
  \[
  f(y) + \frac{\rmLip^\alpha f}{\varepsilon^{1 - \alpha}} d(x, y) \geq
  f(y) + (\rmLip^\alpha f) d(x, y)^\alpha \geq f(x)
  \]
  which implies that the minimum in the definition of
  $f_\varepsilon(x)$ is attained at a point $y \in K$ such that $d(x,
  y) \leq \varepsilon$, and
  \[
  f_\varepsilon(x) = f(y) + \frac{\rmLip^\alpha f}{\varepsilon^{1 -
      \alpha}} d(x, y) \geq f(y) \geq f(x) - (\rmLip^\alpha f)
  \varepsilon^\alpha.
  \]
  This shows that $\|f_\varepsilon - f\|_\infty \leq \rmLip^\alpha
  (f)\varepsilon^\alpha$.

  Let $T = \sum_{k=1}^n a_k \llbracket x_k \rrbracket$ be a
  $0$-polyhedral chain. Then
  \[
  \omega(T) = \sum_{k=1}^n a_k f_\varepsilon(x_k) + \sum_{k=1}^n a_k
  (f(x_k) - f_\varepsilon(x_k)).
  \]
  Approximating $f_\varepsilon$ with a smooth function, we prove that
  \[
  \left| \sum_{k=1}^n a_k f_\varepsilon(x_k) \right| \leq \max\{
  \rmLip f_\varepsilon, \| f_\varepsilon \|_\infty \} \bF(T).
  \]
  Henceforth, if $\varepsilon \leq 1$,
  \begin{align*}
  |\omega(T)| & \leq \max \left\{ \rmLip f_\varepsilon,
  \|f_\varepsilon\|_\infty \right\} \bF(T) + \|f -
  f_\varepsilon\|_\infty \bM(T) \\ & \leq \left( \|f\|_\infty +
  \rmLip^\alpha(f) \right)\left( \frac{1}{\varepsilon^{1 - \alpha}}
  \bF(T) + \varepsilon^\alpha \bM(T) \right).
  \end{align*}
  By choosing $\varepsilon = \bF(T) / \bM(T)$ (which is less than or
  equal to $1$), we obtain
  \[
  \lvert \omega(T)\rvert \leq 2 \left( \| f \|_\infty +
  \rmLip^\alpha(f) \right) \bF(T)^\alpha \bM(T)^{1 - \alpha}.
  \]
  The preceding identity also holds for any chain $T \in \bM_0(K)$, as
  it is the $\bF$-limit of a sequence $(T_n)$ of polyhedral $0$-chains
  such that $\bM(T_n) \uparrow \bM(T)$, and therefore, $\omega \in
  \bCH^{0,\alpha}(K)$. Finally, by the open mapping theorem, $\Gamma$
  is a Banach space isomorphism between $\bCH^{0,\alpha}(K)$ and
  $\rmLip^\alpha(K)$.
\end{proof}

\begin{Corollary}
  \label{cor:CH0alphaHolder}
  The map $\Gamma$ (from~\ref{e:chargesoverRd}(B)) restricts to a
  Fréchet space isomorphism between $\bCH^{0, \alpha}(\R^d)$ and
  $\rmLip^\alpha_{\rmloc}(\R^d)$.
\end{Corollary}

\begin{Empty}[An open problem]
  Proposition~\ref{prop:CH0alphaHolder} and
  Corollary~\ref{cor:CH0alphaHolder} completely describe
  $\alpha$-fractional charges in the case $m = 0$. In the
  zero-codimensional case $m = d$, there is a representation theorem
  for $\alpha$-fractional charges over $[0, 1]^d$ in
  \cite[Theorem~7.2]{Boua}, which characterizes them as the exterior
  derivatives of $\alpha$-Hölder $(d-1)$-forms over $K$ (a similar
  representation theorem for $d$-charges over $\R^d$ can be obtained
  by using partitions of unity).

  Given those results as well as the De Pauw-Moonens-Pfeffer
  representation theorem, one can conjecture that, in the middle cases
  $1 \leq m \leq d-1$, an arbitrary $\alpha$-fractional charge $\omega
  \in \bCH^{m, \alpha}(\R^d)$ can be (non uniquely) decomposed as a
  sum
  \[
  \omega = \eta_1 + \diff \eta_2, \text{ where } \eta_1 \in
  \rmLip^\alpha_{\rmloc}\left(\R^d; \bigwedge^m \R^d\right) \text{ and
  } \eta_2 \in \rmLip^\alpha_{\rmloc}\left(\R^d; \bigwedge^{m-1}
  \R^d\right).
  \]
\end{Empty}

\begin{Proposition}[Compactness]
  Any bounded sequence in $\bCH^{m, \alpha}(K)$ has a subsequence that
  converges in $\bCH^m(K)$ to an $\alpha$-fractional charge.
\end{Proposition}

\begin{proof}
  First observe that $\| \cdot \|_{\bCH^{m, \alpha}(K)}$ (defined on
  $\bCH^m(K)$ with values in $[0, \infty]$) is lower semi-continuous
  with respect to weak convergence. Therefore, we only need to check
  that a sequence $(\omega_n)$ that satisfies
  \[
  M \defeq \sup_n \| \omega_n \|_{\bCH^{m, \alpha}(K)} < \infty
  \]
  has a convergent subsequence in $\bCH^m(K)$. This is an easy
  consequence of the compactness criterion
  (Theorem~\ref{thm:compact}), as for any $T \in \bN_m(K)$, any
  integer $n$ and $\varepsilon > 0$, one has, by Young inequality,
  \[
  |\omega_n(T)| \leq \varepsilon \bN(T) +
  \frac{M^{1/\alpha}}{\varepsilon^{(1-\alpha)/\alpha}}
  \bF(T). \qedhere
  \]
\end{proof}

\begin{Empty}
  We say that a sequence $(\omega_n)$ in $\bCH^{m,\alpha}(K)$
  converges weakly-* to $\omega \in \bCH^{m, \alpha}(K)$ whenever
  \begin{itemize}
  \item[(A)] $(\omega_n)$ is bounded in $\bCH^{m, \alpha}(K)$;
  \item[(B)] $\omega_n \to \omega$ in $\bCH^m(K)$.
  \end{itemize}
  Using the preceding proposition, it is easy to prove that
  condition~(B) can be substituted with ``$\omega_n \to \omega$
  weakly''.

  There is a similar compactness result for charges over
  $\R^d$. Recall that the topology of $\bCH^{m, \alpha}(\R^d)$ is
  induced that the family of seminorms $\| \cdot
  \|_{\bCH^{m,\alpha}(K)}$. A sequence $(\omega_n)$ in
  $\bCH^{m,\alpha}(\R^d)$ is bounded whenever $\sup_n
  \|\omega_n\|_{\bCH^{m,\alpha}(K)} < \infty$ for every compact $K
  \subset \R^d$. It is enough to consider a countable family of
  compact subsets $K$ (for example the closed balls centered at the
  origin with an integer radius). By a diagonal argument, any bounded
  sequence in $\bCH^{m, \alpha}(\R^d)$ has a subsequence that
  converges in $\bCH^{m}(\R^d)$ to some $\alpha$-fractional
  charge. The notion of weak* convergence in $\bCH^{m, \alpha}(\R^d)$
  is easily adapted. We will use the compactness theorem in the
  following form.
\end{Empty}

\begin{Proposition}
  Let $(\omega_n)$ be a sequence in $\bCH^{m, \alpha}(\R^d)$ that is
  bounded and such that $\lim_n \omega_n(T)$ exists for all $T \in
  \bN_m(\R^d)$. Then $(\omega_n)$ converges weakly-* to an
  $\alpha$-fractional charge.
\end{Proposition}

The smoothing of charges provides an example of weak*
convergence. Precise estimates are given in the next proposition.

\begin{Proposition}
  \label{prop:estimatesSmoothing}
  Let $\omega \in \bCH^{m, \alpha}(\R^d)$, let $K \subset \R^d$ be
  compact and $\varepsilon \in \mathopen{(} 0, 1 \mathclose{]}$. We have
  \begin{itemize}
  \item[(A)] $\| \omega * \Phi_\varepsilon \|_{\bCH^{m, \alpha}(K)}
    \leq \| \omega \|_{\bCH^{m, \alpha}(K_\varepsilon)}$;
  \item[(B)] $\| \omega * \Phi_\varepsilon \|_{\bCH^{m, 1}(K)} \leq C
    \varepsilon^{\alpha - 1} \| \omega \|_{\bCH^{m,
        \alpha}(K_\varepsilon)}$;
  \item[(C)] $\| \omega - \omega * \Phi_\varepsilon \|_{\bCH^m(K)}
    \leq C \varepsilon^\alpha \| \omega \|_{\bCH^{m,
        \alpha}(K_\varepsilon)}$.
  \end{itemize}
\end{Proposition}

\begin{proof}
  Let $T \in \bN_m(K)$ be arbitrary. Regarding (A), we have
  \begin{equation}
    \label{eq:proofS}
    |\omega * \Phi_\varepsilon(T)| = |\omega(T * \Phi_\varepsilon)|
    \leq \| \omega \|_{\bCH^{m,\alpha}(K_\varepsilon)} \bN(T *
    \Phi_\varepsilon)^{1-\alpha} \bF(T * \Phi_\varepsilon)^\alpha
  \end{equation}
  because $\rmspt (T * \Phi_\varepsilon) \subset K_\varepsilon$. By
  Proposition~\ref{prop:convT}(C),
  \[
  |\omega * \Phi_\varepsilon(T)| \leq \| \omega
  \|_{\bCH^{m,\alpha}(K_\varepsilon)} \bN(T)^{1 - \alpha}
  \bF(T)^\alpha.
  \]
  Therefore $\|\omega * \Phi_\varepsilon\|_{\bCH^{m, \alpha}(K)} \leq
  \|\omega \|_{\bCH^{m, \alpha}(K_\varepsilon)}$.

  (B) is obtained by combining~\eqref{eq:proofS} with
  Proposition~\ref{prop:convT}(D).

  (C). This time,
  \begin{align*}
    | (\omega - \omega * \Phi_\varepsilon)(T) | & = |\omega(T - T *
    \Phi_\varepsilon)| \\ & \leq \| \omega \|_{\bCH^{m,
        \alpha}(K_\varepsilon)} \bF(T - T * \Phi_\varepsilon)^\alpha
    \bN(T - T * \Phi_\varepsilon)^{1 - \alpha} \\ & \leq C \| \omega
    \|_{\bCH^{m, \alpha}(K_\varepsilon)} \varepsilon^\alpha
    \bN(T)^\alpha \left( \bN(T) + \bN(T * \Phi_\varepsilon) \right)^{1
      - \alpha} & \text{Prop.~\ref{prop:convT}(E)} \\ & \leq C \|
    \omega \|_{\bCH^{m,\alpha}(K_\varepsilon)} \bN(T) &
    \text{Prop.~\ref{prop:convT}(C)}
  \end{align*}
  We conclude with the arbitrariness of $T$.
\end{proof}

We end this section with a (technical) proposition that gives the
$1$-fractional norm of a smooth form.

\begin{Proposition}
  \label{prop:CH1smooth}
  Suppose $\omega \in C^\infty(\R^d; \bigwedge^m \R^d)$ and $K \subset
  \R^d$ is a compact set such that
  \begin{enumerate}
  \item[(A)] Every nonempty open subset of $K$ has positive Lebesgue measure;
  \item[(B)] $K$ is a $1$-Lipschitz retract of $\R^d$.
  \end{enumerate}
  Then $\| \omega \|_{\bCH^{m,1}(K)} = \max \{ \| \omega \|_{\infty,
    K}, \| \diff \omega \|_{\infty, K} \}$.
\end{Proposition}

\begin{proof}
  Hypothesis (B) guarantees that $\bF(T) = \bF_K(T)$ for $T \in
  \bN_m(K)$. Therefore, we consider $A \in \bN_m(K)$ and $B \in
  \bN_{m+1}(K)$ such that $T = A + \partial B$ and compute
  \[
  |\omega(T)| = |A(\omega)| + |B(\diff \omega)| \leq \left( \bM(A) +
  \bM(B) \right) \max \{ \| \omega \|_{\infty, K}, \| \diff \omega
  \|_{\infty, K} \}.
  \]
  Taking the infimum over $A, B$, one obtains that $|\omega(T)| \leq
  \bF(T) \max \{ \| \omega \|_{\infty, K}, \| \diff \omega \|_{\infty,
    K} \}$. This means that $\| \omega \|_{\bCH^{m,1}(K)} \leq \max \{
  \| \omega \|_{\infty, K}, \| \diff \omega \|_{\infty, K} \}$.

  Hypothesis (A) implies that the $\| \cdot \|_{\infty, K}$ seminorms
  can be replaced with essential suprema. One has then
  \begin{gather*}
  \| \omega \|_{\infty, K} = \sup_\zeta \int \langle \omega(x),
  \zeta(x) \rangle \diff x = \sup_\zeta \omega(\calL^d \wedge \zeta)
  \\ \| \diff \omega \|_{\infty, K} = \sup_\xi \int \langle (\diff
  \omega)(x), \xi(x) \rangle \diff x = \sup_{\xi}
  \omega\left(\partial(\calL^d \wedge \xi) \right)
  \end{gather*}
  where $\zeta$ (resp. $\xi$) ranges over the summable
  $m$-vectorfields (resp. $(m+1)$-vectorfields) supported in $K$ of
  $L^1$-norm 1. As $\bF(\calL^d \wedge \zeta) \leq 1$ and
  $\bF\left(\partial(\calL^d \wedge \xi) \right) \leq 1$, one finally
  proves the desired inequality.
\end{proof}

\begin{Corollary}
  \label{cor:411}
  If $\omega \in C^\infty(\R^d, \bigwedge^m \R^d)$, $\eta \in
  C^\infty(\R^d, \bigwedge^{m'} \R^d)$ and $K$ is a compact set that
  satisfies (A) and (B), then $\| \omega \wedge \eta
  \|_{\bCH^{m+m',1}(K)} \leq C \| \omega \|_{\bCH^{m, 1}(K)} \| \eta
  \|_{\bCH^{m', 1}(K)}$, for some constant $C$.
\end{Corollary}

\begin{proof}
  It is a consequence of the identity $\diff (\omega \wedge \eta) =
  \diff \omega \wedge \eta + (-1)^m \omega \wedge \diff \eta$. The
  reader interested in computing the constant may consult
  \cite[1.8.1]{Fede}.
\end{proof}

\section{A Littlewood-Paley type decomposition of fractional charges}
\label{sec:LPdec}

We wish to introduce a decomposition result for $\alpha$-fractional
charges into much more regular components (that are at least
$1$-fractional). It will be analogous to the decomposition of Hölder
functions that is well described in \cite[Appendix~B,~2.6]{Grom} that
we briefly recall here. Any $\alpha$-Hölder continuous function $f
\colon X \to \R$ (for $0 < \alpha < 1$) defined on a metric space can
be decomposed into Lipschitz parts $f = \sum_{n=0}^\infty f_n$, where,
for all $n \geq 0$,
\begin{itemize}
\item[(A)] $\|f_n\|_\infty \leq C 2^{-n \alpha}$;
\item[(B)] $\rmLip f_n \leq C 2^{n(1 - \alpha)}$;
\item[(C)] $\sum_{n=0}^\infty f_n(x)$ converges for some $x \in X$.
\end{itemize}
Here $C$ is a constant independent of $n$.  We can think of $f_n$ as
being the part of $f$ whose frequencies are localized around
$2^n$. The estimates (A) and (B) guarantee that $f_n$ is in
$\rmLip^\alpha(X)$, with $\rmLip^\alpha f_n \leq 2C$. In our future
decomposition of charges, the smooth components will be flat
cochains. We begin our study when the domain is a compact subset $K$
of $\R^d$, with an elementary lemma, that estimates the
$\alpha$-fractional norm of such a component.

\begin{Lemma}
  Let $\omega \in \bCH^{m,1}(K)$, $C \geq 0$ and $\varepsilon >
  0$. Suppose that 
  \begin{equation}
    \label{eq:Semmes}
    \|\omega\|_{\bCH^{m,1}(K)} \leq \frac{C}{\varepsilon^{1-\alpha}}
    \text{ and } \|\omega\|_{\bCH^m(K)} \leq C \varepsilon^\alpha.
  \end{equation}
  Then $\|\omega\|_{\bCH^{m, \alpha}(K)} \leq C$.
\end{Lemma}

\begin{proof}
  For any $T \in \bN_m(K)$,
  \begin{align*}
  |\omega(T)| & \leq |\omega(T)|^\alpha |\omega(T)|^{1 - \alpha} \\ &
  \leq \left( \frac{C}{\varepsilon^{1 - \alpha}} \right)^{\alpha} C^{1
    - \alpha} \varepsilon^{\alpha(1-\alpha)} \bF(T)^{\alpha} \bN(T)^{1
    - \alpha} \\ & \leq C \bF(T)^\alpha \bN(T)^{1 - \alpha}. \qedhere
  \end{align*}
\end{proof}

\begin{Proposition}
  \label{prop:chLP}
  Suppose $0 < \alpha < 1$. Let $(\omega_n)$ be a sequence in
  $\bCH^{m,1}(K)$ such that
  \[
  \|\omega_n\|_{\bCH^{m,1}(K)} \leq C 2^{n(1 - \alpha)} \text{ and }
  \|\omega_n\|_{\bCH^m(K)} \leq \frac{C}{2^{n\alpha}}
  \]
  for some $C \geq 0$ and for all $n$. Then $\sum_{n=0}^\infty
  \omega_n$ converges weakly-* to a charge in $\bCH^{m,\alpha}(K)$ and
  \[
  \left\| \sum_{n=0}^\infty \omega_n \right\|_{\bCH^{m,\alpha}(K)} \leq
  C_{\ref{prop:chLP}} C
  \]
  where $C_{\ref{prop:chLP}} = C_{\ref{prop:chLP}}(\alpha)$ is a
  constant.
\end{Proposition}

Note that the convergence of the series $\sum_{n=0}^\infty \omega_n$
is not strong in $\bCH^{m,\alpha}(K)$, but only weak*, a fact that
should evoke some sort of orthogonality of the components $\omega_n$.

\begin{proof}
  Let $T \in \bN_m(K)$ be nonzero and $p \geq 1$. We have
  \[
  \left\lvert \sum_{n=0}^{p-1} \omega_n(T) \right\rvert
  \leq C \bF(T) \sum_{n=0}^{p-1} 2^{n(1-\alpha)} \leq \frac{C 2^{p(1 -
      \alpha)}}{2^{1-\alpha} - 1} \bF(T).
  \]
  On the other hand,
  \[
  \left\lvert \sum_{n=p}^{\infty} \omega_n(T) \right\rvert \leq C
  \bN(T) \sum_{n=p}^\infty 2^{-n\alpha} \leq \frac{C 2^{-p\alpha}}{1
    - 2^{-\alpha}} \bN(T)
  \]
  Now, choose $p \geq 1$ such that
  \[
  2^{p-1} \leq \frac{\bN(T)}{\bF(T)} \leq 2^{p}.
  \]
  This is possible because $\bN(T) \geq \bF(T)$.  Then
  \[
  \left\lvert \sum_{n=0}^{\infty} \omega_n(T) \right\rvert \leq \left(
  \frac{2^{1 - \alpha}}{2^{1-\alpha} - 1} + \frac{1}{1 -
    2^{-\alpha}}\right) C \bF(T)^\alpha \bN(T)^{1 - \alpha}
  \]
  This shows that $\sum_{n=0}^\infty \omega_n \in
  \bCH^{m,\alpha}(K)$. The preceding argument shows that the sequence
  of partial sums $\left( \sum_{n=0}^p \omega_n \right)_{p \geq 0}$ is
  bounded in $\bCH^{m,\alpha}(K)$, and from there, it is easy to see
  that the convergence is weak-*.
\end{proof}

When adapted to charges over $\R^d$, the preceding proposition takes
the following form.

\begin{Corollary}
  \label{cor:LP}
  Suppose $0 < \alpha < 1$. Let $(\omega_n)$ be a sequence in
  $\bCH^{m,1}(\R^d)$ such that, for each compact $K \subset \R^d$,
  there is $C_K \geq 0$ such that
  \[
  \|\omega_n\|_{\bCH^{m,1}(K)} \leq C_K 2^{n(1 - \alpha)} \text{ and }
  \|\omega_n\|_{\bCH^m(K)} \leq \frac{C_K}{2^{n\alpha}}
  \]
  for all $n$. Then $\sum_{n=0}^\infty \omega_n$ converges weakly-* to
  a charge in $\bCH^{m,\alpha}(\R^d)$.
\end{Corollary}

\begin{Empty}
  Of course, it is enough to check the hypothesis of
  Corollary~\ref{cor:LP} when $K$ ranges over (non degenerate) closed
  balls. A Littlewood-Paley type decomposition of a fractional charge
  $\omega \in \bCH^{m, \alpha}(\R^d)$ can be obtained by convolution
  \[
  \omega = \omega * \Phi_{1} + \sum_{n=0}^\infty \left( \omega *
  \Phi_{2^{-(n+1)}} - \omega * \Phi_{2^{-n}} \right).
  \]
  We claim that the weak* convergence above is ensured by
  Corollary~\ref{cor:LP} and the estimates of
  Proposition~\ref{prop:estimatesSmoothing}. This decomposition will
  play a pivotal role in the forthcoming proof of
  Theorem~\ref{thm:main}.
\end{Empty}

\section{Main result}

\begin{Theorem}
  \label{thm:main}
  Let $\alpha, \beta$ be parameters such that $0 < \alpha, \beta \leq
  1$ and $\alpha + \beta > 1$. There is a unique map
  \[
  \wedge \colon \bCH^{m, \alpha}(\R^d) \times \bCH^{m', \beta}(\R^d)
  \to \bCH^{m+m', \alpha + \beta - 1}(\R^d)
  \]
  such that
  \begin{enumerate}
  \item[(A)] $\wedge$ extends the pointwise exterior product between smooth forms;
  \item[(B)] Weak*-to-weak* continuity: if $(\omega_n)$ and $(\eta_n)$
    are two sequences that converge weakly-* to $\omega$ and $\eta$ in
    $\bCH^{m, \alpha}(\R^d)$ and $\bCH^{m', \beta}(\R^d)$
    respectively, then $\omega_n \wedge \eta_n$ converge weakly-* to
    $\omega \wedge \eta$ in $\bCH^{m+m', \alpha + \beta - 1}(\R^d)$.
  \end{enumerate}
  Moreover, $\wedge$ is continuous.
\end{Theorem}

\begin{proof}
  We proved in Proposition~\ref{prop:estimatesSmoothing} that $\omega
  * \Phi_\varepsilon \to \omega$ weakly-* in $\bCH^{m, \alpha}(\R^d)$,
  and similarly, $\eta * \Phi_\varepsilon \to \eta$ weakly-* in
  $\bCH^{m', \beta}(\R^d)$. Taking into account that $\omega *
  \Phi_\varepsilon$ and $\eta * \Phi_\varepsilon$ are smooth, the
  uniqueness of the map $\wedge$ follows.

  We now address the issue of existence. First we treat the case
  $(\alpha, \beta) \neq (1, 1)$. Abbreviate $\omega_n = \omega *
  \Phi_{2^{-n}}$ and $\eta_n = \eta * \Phi_{2^{-n}}$. We shall prove
  that the weak* limit of $(\omega_n \wedge \eta_n)$ exists in
  $\bCH^{m+m', \alpha + \beta - 1}(\R^d)$. This will be achieved if we
  manage to prove that the following series converges weakly-*
  \begin{equation}
    \label{eq:seriesW*}
    \omega_0 \wedge \eta_0 + \sum_{n=0}^\infty \left( \omega_{n+1}
    \wedge (\eta_{n+1} - \eta_n) + (\omega_{n+1} - \omega_n) \wedge
    \eta_n \right)
  \end{equation}
  in $\bCH^{m+m', \alpha + \beta - 1}(\R^d)$. Let us note in passing
  that both sums $\sum_{n=0}^\infty \omega_{n+1} \wedge (\eta_{n+1} -
  \eta_n)$ and $\sum_{n=0}^\infty (\omega_{n+1} - \omega_n) \wedge
  \eta_n$ can be interpreted as paraproducts.  The weak* convergence
  will follow from an application of Corollary~\ref{cor:LP}. To this
  end, we let $K$ be a closed (non degenerate) ball and estimate
  \[
  \| \eta_{n+1} - \eta_n \|_{\bCH^{m'}(K)} \leq \| \eta_{n+1} - \eta
  \|_{\bCH^{m'}(K)} + \| \eta_n - \eta \|_{\bCH^{m'}(K)} \leq C 2^{-n\beta}
  \| \eta \|_{\bCH^{m', \beta}(K_1)}
  \]
  by Proposition~\ref{prop:estimatesSmoothing} and similarly
  \[
  \| \omega_{n+1} - \omega_n \|_{\bCH^m(K)} \leq C 2^{-n\alpha}
  \|\omega\|_{\bCH^{m, \alpha}(K_1)}.
  \]
  Next we wish to control the $\bCH^{m+m'}(K)$ seminorm of the
  exterior product $\omega_{n+1} \wedge (\eta_{n+1} - \eta_n)$. Let $T
  \in \bN_{m+m'}(K)$. We recall that
  \[
  \partial (T \hel \omega_{n+1}) = (-1)^m (\partial T) \hel
  \omega_{n+1} + (-1)^{m+1} T \hel \diff \omega_{n+1}.
  \]
  Hence
  \begin{align*}
    \bN(T \hel \omega_{n+1}) & \leq C \bN(T) \max \{
    \|\omega_{n+1}\|_{K, \infty}, \| \diff \omega_{n+1} \|_{K, \infty}
    \} \\ & \leq C \bN(T) \| \omega_{n+1} \|_{\bCH^{m, 1}(K)} &
    \text{by Proposition~\ref{prop:CH1smooth}} \\ & \leq C \bN(T) \|
    \omega \|_{\bCH^{m, \alpha}(K_1)} 2^{n(1 - \alpha)} & \text{by
      Proposition~\ref{prop:estimatesSmoothing}(B)}
  \end{align*}
  We deduce that
  \begin{align*}
    \left|\omega_{n+1} \wedge (\eta_{n+1} - \eta_n) (T)\right| & =
    \left| (\eta_{n+1} - \eta_n)(T \hel \omega_{n+1}) \right| \\ &
    \leq \| \eta_{n+1} - \eta_n \|_{\bCH^{m'}(K)} \bN(T \hel
    \omega_{n+1}) \\ & \leq C \|\omega\|_{\bCH^{m,\alpha}(K_1)} \|
    \eta \|_{\bCH^{m',\beta}(K_1)} 2^{n(1 - \alpha - \beta)} \bN(T).
  \end{align*}
  As a result,
  \[
  \|\omega_{n+1} \wedge (\eta_{n+1} - \eta_n) \|_{\bCH^m(K)} \leq C
  \|\omega\|_{\bCH^{m,\alpha}(K_1)} \| \eta \|_{\bCH^{m', \beta}(K_1)}
  2^{n(1 - \alpha - \beta)}.
  \]
  Next we estimate the $\bCH^{m,1}(K)$ seminorm of $\omega_{n+1}
  \wedge (\eta_{n+1} - \eta_n)$. By
  Proposition~\ref{prop:estimatesSmoothing}(B),
  \begin{gather*}
    \|\omega_{n+1}\|_{\bCH^{m,1}(K)} \leq C 2^{n(1 - \alpha)} \| \omega
    \|_{\bCH^{m, \alpha}(K_1)} \\ \| \eta_{n+1} - \eta_n \|_{\bCH^{m',
        1}(K)} \leq \|\eta_{n+1}\|_{\bCH^{m', 1}(K_1)} + \|\eta_n
    \|_{\bCH^{m', 1}(K)} \leq C 2^{n(1 - \beta)} \| \eta \|_{\bCH^{m',
        \beta}(K_1)}
  \end{gather*}
  By Corollary~\ref{cor:411},
  \[
  \|\omega_{n+1} \wedge (\eta_{n+1} - \eta_n) \|_{\bCH^{m,1}(K)} \leq
  C \| \omega \|_{\bCH^{m,\alpha}(K_1)} \|\eta\|_{\bCH^{m',
      \beta}(K_1)} 2^{n(1 - (\alpha + \beta - 1))}.
  \]
  Similarly, we have
  \begin{gather*}
    \| (\omega_{n+1} - \omega_n) \wedge \eta_n \|_{\bCH^{m+m'}(K)}
    \leq C \|\omega\|_{\bCH^{m,\alpha}(K_1)} \| \eta
    \|_{\bCH^{m',\beta}(K_1)} 2^{n(1 - \alpha - \beta)} \\ \| (\omega_{n+1}
    - \omega_n) \wedge \eta_n \|_{\bCH^{m+m', 1}(K)} \leq C \| \omega
    \|_{\bCH^{m,\alpha}(K_1)} \|\eta\|_{\bCH^{m', \beta}(K_1)} 2^{n(1
      - (\alpha + \beta - 1))}
  \end{gather*}
  Thus the series in~\eqref{eq:seriesW*} converges weakly-* and we
  naturally define $\omega \wedge \eta$ to be the weak* limit of
  $(\omega_n \wedge \eta_n)$. By Proposition~\ref{prop:chLP}, we have
  \begin{equation}
    \label{eq:wedgeCont}
  \| \omega \wedge \eta \|_{\bCH^{m, \alpha + \beta - 1}(K)} \leq C \|
  \omega \|_{\bCH^{m,\alpha}(K_1)} \| \eta \|_{\bCH^{m', \beta}(K_1)}.
  \end{equation}
  Thus, $\wedge$ is continuous, as desired.

  Now we need to prove that (A) and (B) hold. (A) is easy, for if
  $\omega$ and $\eta$ are already smooth, then $(\omega_n \wedge
  \eta_n)$ converges locally uniformly (and thus weakly) to the
  pointwise exterior product of $\omega$ and $\eta$. Then the weak*
  and weak limits coincide, so we can conclude that $\omega \wedge
  \eta$ has its natural meaning.

  Finally we prove (B). Let $(\omega^{(p)})$ and $(\eta^{(p)})$ two
  sequences, indexed by $p \geq 0$, that converge weakly-* towards
  $\omega \in \bCH^{m, \alpha}(\R^d)$ and $\eta \in \bCH^{m',
    \beta}(\R^d)$ as $p \to \infty$. As before, we fix a closed ball
  $K$ and we set $\omega_n^{(p)} = \omega^{(p)} * \Phi_{2^{-n}}$ and
  $\eta_n^{(p)} = \eta^{(p)} * \Phi_{2^{-n}}$ for all $n,
  p$. Inequality~\eqref{eq:wedgeCont}, applied to $\omega^{(p)}$ and
  $\eta^{(p)}$, show that the sequence $(\omega^{(p)} \wedge
  \eta^{(p)})$ is bounded in $\bCH^{m+m', \alpha + \beta - 1}(\R^d)$.

  Proposition~\ref{prop:smooth} entails that, for all integer $n$, the
  smooth forms $\omega_n^{(p)}$ and $\eta_n^{(p)}$ converge locally
  uniformly to $\omega_n$ and $\eta_n$ as $p \to \infty$. Thus, for a
  fixed normal current $T$ with support in a non degenerate closed
  ball, we have
  \begin{multline*}
  T\left( \omega_{n+1} \wedge (\eta_{n+1} - \eta_n) + (\omega_{n+1} -
  \omega_n) \wedge \eta_n \right) = \\
  \lim_{p \to \infty} T\left(
  \omega_{n+1}^{(p)} \wedge (\eta_{n+1}^{(p)} - \eta_n^{(p)}) +
  (\omega_{n+1}^{(p)} - \omega_n^{(p)}) \wedge \eta_n^{(p)} \right).
  \end{multline*}
  Likewise,
  \[
  T(\omega_0 \wedge \eta_0) = \lim_{p \to \infty} T(\omega_0^{(p)}
  \wedge \eta_0^{(p)}).
  \]
  Arguing as before, one has
  \begin{multline*}
  \left| T \left( \omega_{n+1}^{(p)} \wedge (\eta_{n+1}^{(p)} -
  \eta_n^{(p)}) + (\omega_{n+1}^{(p)} - \omega_n^{(p)}) \wedge
  \eta_n^{(p)} \right) \right| \\ \leq C \|\omega^{(p)}\|_{\bCH^{m,
      \alpha}(K_1)} \| \eta^{(p)} \|_{\bCH^{m', \beta}(K_1)} 2^{n(1 -
    \alpha - \beta)} \bN(T).
  \end{multline*}
  As the sequences $(\omega^{(p)})$ and $(\eta^{(p)})$ are bounded,
  the previous bound can be made uniform in $p$, allowing us to apply
  Lebesgue's dominated convergence theorem in
  \begin{multline*}
    \omega^{(p)} \wedge \eta^{(p)}(T) \\
    \begin{aligned}
      & = \lim_{p \to \infty} \left(T(\omega_0^{(p)} \wedge
      \eta_0^{(p)}) + \sum_{n=0}^\infty T\left( \omega_{n+1}^{(p)}
      \wedge (\eta_{n+1}^{(p)} - \eta_n^{(p)}) + (\omega_{n+1}^{(p)} -
      \omega_n^{(p)}) \wedge \eta_n^{(p)} \right) \right)\\ & = \omega
      \wedge \eta(T)
    \end{aligned}
  \end{multline*}

  The case $\alpha = \beta = 1$, though simpler, requires special
  attention. The uniqueness of $\wedge$ is already established. As
  previously, we define $\omega \wedge \eta$ to be the weak* limit
  $\lim_n \omega_n \wedge \eta_n$. It exists because of the
  compactness theorem. Indeed, for a suitable $K$, we have, by
  Corollary~\ref{cor:411} and
  Proposition~\ref{prop:estimatesSmoothing}(A)
  (or~\ref{prop:estimatesSmoothing}(B))
  \begin{align*}
    \| \omega_n \wedge \eta_n \|_{\bCH^{m+m', 1}(K)} & \leq C
    \|\omega_n\|_{\bCH^{m,1}(K)} \| \eta_n \|_{\bCH^{m',1}(K)} \\ &
    \leq C \| \omega \|_{\bCH^{m,1}(K_1)} \| \eta \|_{\bCH^{m',
        1}(K_1)}
  \end{align*}
  It remains to show that $(\omega_n \wedge \eta_n)$ converges
  weakly. This is the case because, for $T \in \bN_m(K)$,
  \begin{equation}
    \label{eq:cas11}
  T(\omega_N \wedge \eta_N) = T(\omega_0 \wedge \eta_0) +
  \sum_{n=0}^{N-1} T\left( \omega_{n+1} \wedge (\eta_{n+1} - \eta_n) +
  (\omega_{n+1} - \omega_n) \wedge \eta_n \right)
  \end{equation}
  As before, we establish
  \[
  \left| T\left( \omega_{n+1} \wedge (\eta_{n+1} - \eta_n) +
  (\omega_{n+1} - \omega_n) \wedge \eta_n \right) \right| \leq \frac{C
    \| \omega \|_{\bCH^{m, 1}(K_1) } \| \eta \|_{\bCH^{m', 1}(K_1)}
  }{2^n}
  \]
  which ensures the absolute convergence of the series in~\eqref{eq:cas11}.

  Now that the exterior product is well-defined as a map
  $\bCH^{m,1}(\R^d) \times \bCH^{m', 1}(\R^d) \to \bCH^{m+m',
    1}(\R^d)$, properties (A) and (B) are shown as before.
\end{proof}

\begin{Empty}[Properties of $\wedge$]
  Using the weak* density of smooth forms, it is easy to extend the
  well-known formulae of exterior calculus to fractional
  charges. Among them, we have, for $\omega \in \bCH^{m,
    \alpha}(\R^d)$, $\eta \in \bCH^{m', \beta}(\R^d)$ and $\alpha +
  \beta > 1$,
  \begin{itemize}
  \item[(A)] $\wedge$ is bilinear;
  \item[(B)] $\omega \wedge \eta = (-1)^{mm'} \eta \wedge \omega$;
  \item[(C)] $\diff (\omega \wedge \eta) = \diff \omega \wedge \eta +
    (-1)^m \omega \wedge \diff \eta$.
  \end{itemize}
  The last item uses the weak*-to-weak* continuity of the exterior
  derivative.

  Regarding the product of many fractional charges, we notice that
  $\omega_1 \wedge \cdots \wedge \omega_k$ makes sense (and the
  product is associative) as long as $\omega_i$ is
  $\alpha_i$-fractional and $\alpha_1 + \cdots + \alpha_k > k-1$. In
  this case, the result is an $(\alpha_1 + \cdots + \alpha_k -
  (k-1))$-fractional charge.
\end{Empty}

\bibliographystyle{amsplain} \bibliography{phil.bib}

\end{document}